\documentclass[12pt]{amsart}  

\usepackage{custom_preamble,bbm} 

\newcounter{pblm}

\theoremstyle{definition}

\begin{document}

\title{The stability of finite sets in dyadic groups}

\begin{abstract}
We show that there is an absolute $c>0$ such that any subset of $\F_2^\infty$ of size $N$ is $O(N^{1-c})$-stable in the sense of Terry and Wolf.  By contrast a size $N$ arithmetic progression in $\Z$ is not $N$-stable.
\end{abstract}

\author{\tsname}
\address{\tsaddress}
\email{\tsemail}

\maketitle

\setcounter{section}{1}

Arithmetic stability is a concept introduced by Terry and Wolf in \cite{terwol::0} which quickly leads to many lines of questions. It is the purpose of this note to highlight one such line.

Formally, suppose that $G$ is an Abelian group and $A \subset G$.  For $k \in \N$ we say $A$ has the \textbf{$k$-order property (in $G$)} if there are vectors $s,t \in G^k$ such that $s_i +t_j \in A$ if and only if $i \leq j$ -- we say that $s$ and $t$ \textbf{witness} the $k$-order property -- and it is \textbf{$k$-stable} if it does not have the $k$-order property.

The definition of the $k$-order property is inspired by the model-theoretic order property of formulas \cite[Definition 2.2]{she::0}.  Indeed, in the setup of \cite{she::0}, the formula $x+y \in A$ has the order property if and only if $A$ has the $k$-order property for all $k \in \N$.

The model-theoretic stability of a formula has many equivalent definitions (see \cite[\S2]{she::0}) one of which is the negation of the order property.  (Hence the terminology above.)  It also extends to theories; a simple example of a stable theory is the theory of $(G,+,0)$ when $G$ is a direct sum of a divisible group and a group of bounded order \cite[Theorem 1]{mac::1}\footnote{And the fact that totally transcendental theories are stable.}.

The problem of determining when stability is preserved under the adjoining of a predicate was opened up in \cite{caszie::}, and the particular case of adjoining sets of integers as unary predicates to the theory of $(\Z,+,0)$ has been pursued recently in \cite{palskl::,con::4} and \cite{con::2}.  Adjoining the natural numbers makes the theory unstable, and the present work might be seen as part of a quantitative investigation into whether there are subsets of dyadic groups that behave as badly as the naturals do in the integers.

We say that sets $S$ and $T$ \textbf{witness} the $k$-order property if there is an enumeration of $S$ as $s_1,\dots,s_k$ and $T$ as $t_1,\dots, t_k$ such that $s$ and $t$ witness the $k$-order property.  It is useful to observe that if $s,t \in G^k$ witness the $k$-order property in $A$ and $\sigma$ and $\tau$ are permutations of $\{1,\dots,k\}$ such that $s':=(s_{\sigma(i)})_i$ and $t':=(t_{\tau(j)})_j$ also witness the $k$-order property in $A$, then $\sigma$ and $\tau$ are both the identity permutation.\footnote{Indeed, since $\tau$ is a surjection and $s$ and $t$ witness the $k$-order property the set $\{\tau(j): s_{\sigma(i)}+t_{\tau(j)} \in A\}$ has size $k+1-\sigma(i)$.  On the other hand since $s'$ and $t'$ witness the $k$-order property the set $\{j:s_{\sigma(i)}+t_{\tau(j)} \in A\}=\{j: s'_i+t'_j \in A\}$ has size $k+1-i$.  Since $\tau$ is a bijection $k+1-\sigma(i)=k+1-i$ and so $\sigma(i)=i$, and the same arguments work for $\tau$.} In particular, if $s,t \in G^k$ witness the $k$-order property then the $s_i$s are all distinct (and similarly for the $t_j$s); and if sets $S$ and $T$ witness the $k$-order property then the associated enumerations are unique.

We begin with the following simple proposition.
\begin{proposition}\label{prop.upper}
Suppose that $A \subset G$ has size $N$.  Then $A$ is $(N+1)$-stable.
\end{proposition}
\begin{proof}
Suppose that $s$ and $t$ witness the $k$-order property in $G$.  Since the $t_j$'s are distinct, $\{s_1+t_j: 1\leq j \leq k\}$ is a set of $k$ elements contained in $A$ so $k \leq N$, and the result is proved.
\end{proof}
In some groups this is best possible.
\begin{proposition}\label{prop.ap}
Suppose that $A$ is an arithmetic progression of length $N$ in $\Z$.  Then $A$ is not $N$-stable.
\end{proposition}
\begin{proof}
Write $A=\{x,x+d,\dots,x+(N-1)d\}$, and let $s_i:=x-id$ and $t_i=id$ for $1 \leq i \leq N$.  A simple check shows $s,t \in \Z^N$ witness the $N$-order property. (\emph{c.f.} \cite[Lemma 6.3]{sis::1}.)
\end{proof}
Given this we ask what happens in groups not containing long arithmetic progressions.  The dyadic group $\F_2^\infty$ (the direct sum of $\F_2$ with itself countably many times) is the prototypical example of such a group.  Its study has been tremendously useful in additive combinatorics, and we refer the reader to the survey \cite{gre::9} by Green and the sequel \cite{wol::3} by Wolf for more information.
\begin{theorem}\label{thm.main}
There is $c>0$ such that every $A \subset \F_2^\infty$ of size $N$ is $O(N^{1-c})$-stable.
\end{theorem}
\cite[Example 3]{terwol::0} shows that $c$ cannot be improved past $\frac{1}{2}$; we shall develop this to give the following.
\begin{proposition}\label{prop.lower}
There is $c>0$ such that for all $N \in \N$ there is a set $A \subset \F_2^\infty$ of size $N$ that is not $\Omega(N^{\frac{1}{2}+c})$-stable.
\end{proposition}
Before turning to the proofs we make two remarks.  First, the notion of stability is of interest in the papers \cite{terwol::0,terwol::1} of Terry and Wolf when a set has small stability, and it is not clear to us whether our work, which sits at the other end of the scale, genuinely gets at mathematically significant issues.

Secondly, it may well be that there is a rather more direct combinatorial argument and allied lower-bound construction that significantly simplifies and strengthens the work here; certainly we do not know how to rule out such a possibility.

Our proof of Theorem \ref{thm.main} decouples into two parts.  We begin with the `modelling' part.
\begin{lemma}\label{lem.deco}
Suppose that $A \subset \F_2^\infty$ has the $k$-order property; $|A| \leq Kk$; and $1 \leq l < \frac{1}{4}k$ is an integer with $l=\eta k$.  Then there is a natural number $n$ and set $A' \subset \F_2^n$ with $2^n=O(\eta^{-10}K^{15}k)$ that has the $(k-2l+1)$-order property.
\end{lemma}
It is tremendously tempting to note that the Balog-Szemer{\'e}di-Gowers theorem \cite[Theorem 2.29]{taovu::} applies to show that if $A \subset \F_2^\infty$ has the $k$-order property (witnessed by sets $S$ and $T$) and $|A| \leq Kk$ then there are sets $S' \subset S$ and $T' \subset T$ with $|S'|,|T'| = \Omega(k)$ and $|S'+T'| =O(K^{3}k)$.  While this seems very well suited, in fact we shall find it easier to proceed directly.
\begin{proof}
Let $s,t \in (\F_2^\infty)^k$ witness the $k$-order property in $A$, and let $G$ be the group generated by $A\cup\{t_1,\dots,t_k,s_1,\dots,s_k\}$ (which is finite) and recall that the $s_i$'s and $t_j$'s are distinct.  Put
\begin{equation*}
S':=\left\{s_i: 1 \leq i \leq l\right\} \text{ and } S^+:=\left\{s_i : l \leq i \leq k-l\right\},
\end{equation*}
and
\begin{equation*}
T':=\left\{t_j: k-l\leq j \leq k\right\} \text{ and }T^+:=\left\{t_j: l \leq j \leq k-l\right\}.
\end{equation*}
It follows that
\begin{equation*}
|S'|,|T'| \geq l, |S^+|,|T^+| \geq k-2l \text{ and } S'+T^+,S'+T',S^++T' \subset A.
\end{equation*}
The Ruzsa triangle inequality \cite[Lemma 2.6]{taovu::} then gives
\begin{align*}
|S^++T^+|& \leq \frac{|S^++T'||-T'+T^+|}{|-T'|} = \frac{|S^++T'||T'-T^+|}{|T'|}\\ & \leq \frac{|S^++T'||T'+S'||-S'-T^+|}{|T'||-S'|} = \frac{|S^++T'||T'+S'||S'+T^+|}{|T'||S'|}\\
& \leq \eta^{-2}K^3k \leq 2\eta^{-2}K^3\min\{|T^+|,|S^+|\}.
\end{align*}
We now follow the proof of \cite[Proposition 6.1]{greruz::0}.  Let $n \in \N$ be the smallest positive integer for which there is a homomorphism $\phi:G \rightarrow \F_2^n$ such that $\phi$ restricted to $S^++T^+$ is injective.  Such exists since $S^++T^+$ is finite, so projection onto the (finite) group it generates is an example.

Suppose there is some $x \in \F_2^n\setminus (\phi(S^+)-\phi(S^+) +\phi(T^+)-\phi(T^+))$.  Let $\theta:\F_2^n \rightarrow \F_2^{n-1}$ be a homomorphism with $\ker \theta=\{0,x\}$.  Then $\theta \circ \phi:G \rightarrow \F_2^{n-1}$ is a homomorphism and so by minimality of $n$ there are elements $s,s' \in S^+$ and $t,t' \in T^+$ such that
\begin{equation*}
\theta\circ \phi(s+t)=\theta \circ \phi(s'+t')\text{ and } s+t \neq s'+t'.
\end{equation*}
It follows that
\begin{equation*}
\phi((s+t)-(s'+t')) \in\ker \theta=\{0,x\}.
\end{equation*}
Since $s +t \neq s'+t'$ we see $\phi(s)+\phi(t)-\phi(s')-\phi(t') =x$ which contradicts how $x$ was chosen.  It follows that
\begin{equation*}
2^n = |\phi(S^+)-\phi(S^+) +\phi(T^+)-\phi(T^+)|.
\end{equation*}
Let $\emptyset \neq Z \subset T^+$ be such that $|Z|^{-1}|S^++Z|$ is minimal so that
\begin{equation*}
\frac{|S^++Z|}{|Z|}=\min\left\{\frac{|S^++Z'|}{|Z'|}: \emptyset \neq Z' \subset Z\right\},
\end{equation*}
and similarly for $\emptyset \neq W \subset S^+$.  By \cite[Proposition 2.1]{pet::} (and the fact that $-A=A$ for all sets in $G$ and $\phi$ is injective on $S^++T^+$) we see that
\begin{align*}
2^n & = |\phi(S^+)-\phi(S^+) +\phi(T^+)-\phi(T^+)|\\
& = \frac{|S^+-S^++T^+-T^++\ker \phi|}{|\ker \phi|}\\
& = \frac{|2S^++2T^++\ker \phi|}{|\ker \phi|}\\
& \leq  \frac{|2S^++2T^++Z+W+\ker \phi|}{|\ker \phi|}\\
& \leq (2\eta^{-2}K^3)^4\frac{|Z+W+\ker \phi|}{|\ker\phi|}\\
& \leq (2\eta^{-2}K^3)^4\frac{|S^++T^++\ker \phi|}{|\ker \phi|} \leq 2^{4}\eta^{-10}K^{15}k.
\end{align*}
Finally, let
\begin{equation*}
A':=\left\{\phi(s_i+t_j): l \leq i \leq j \leq k-l\right\},
\end{equation*}
and let $s',t' \in (\F_2^m)^{k-2l+1}$ be defined by
\begin{equation*}
s'_i:=\phi\left(s_{l+i-1}\right) \text{ and } t'_j:=\phi\left(t_{l+j-1}\right) \text{ for }1 \leq i,j \leq k-2l+1.
\end{equation*}
Since $\phi$ is injective we see that $s_i'+t_j' \in A'$ if and only if $s_{l+i-1}+t_{l+j-1} \in A$, which is true if and only if $l+i-1 \leq l+j-1$, which in turn is true if and only if $i \leq j$.  We conclude that $A'$ has the $(k-2l+1)$-order property and we are done.
\end{proof}
The second ingredient is the following.
\begin{lemma}\label{lem.dense}
There is $c>0$ such that every $A \subset \F_2^n$ is $O(2^{n(1-c)})$-stable.
\end{lemma}
We shall use the polynomial method from the work of Croot, Lev and Pach \cite{crolevpac::}, though we follow the sequel \cite{ellgij::0} by Ellenberg and Gijswijt.

We write $S_n$ for the $\F_2$-vector space of maps $\F_2^n \rightarrow \F_2$; put
\begin{equation*}
M_n^d:=\left\{\F_2^n \rightarrow \F_2; x \mapsto \prod_{i \in I}{x_i}: I \subset [n] \text{ has } |I| \leq d\right\},
\end{equation*}
so that $M_n^d$ is linearly independent; let $S_n^d$ be the subspace of $S_n$ with $M_n^d$ as a basis; and write $M_n:=M_n^n$ for the basis of monomials for $S_n$, so
\begin{equation*}
\dim S_n^d=\sum_{r=0}^d{\binom{n}{r}} = \sum_{r=n-d}^n{\binom{n}{r}}.
\end{equation*}
This can be estimated using $H$, the binary entropy function -- we refer to \cite[\S11, Chapter 10]{macslo::0} for the relevant estimate.\footnote{For completeness $H:[0,1] \rightarrow \R; p \mapsto -p\log_2 p -(1-p)\log_2 (1-p)$ with the usual conventions that $H(0)=H(1)=0$, and \cite[Lemma 8, \S11, Chapter 10]{macslo::0} tells us that
\begin{equation}\label{eqn.entropy}
\sum_{r=pn}^{n}{\binom{n}{r}} \leq 2^{H(p)n} \text{ whenever }p \in \left(\frac{1}{2},1\right] \text{ and }pn \in \Z.
\end{equation}}
We turn now to the proof.
\begin{proof}
Suppose that $k$ is maximal such that $A$ has the $k$-order property; $s,t \in (\F_2^n)^k$ witness the $k$-order property in $A$; and write $S:=\{s_i: 1\leq i \leq k\}$ and $T:=\{t_j: 1 \leq j \leq k\}$. 
Suppose that there are elements $1 \leq i,j \leq k$ such that $s_i+t_i = s_j+t_j$.  Without loss of generality we have $i \leq j$ and hence (since $2\cdot t_i=0_{\F_2^n}$ and $-t_j=t_j$)
\begin{equation}\label{eqn.problem}
s_j+t_i = s_j-t_i= s_i-t_j=s_i+t_j \in A.
\end{equation}
It follows that $i=j$, and so $\{s_i+t_i : 1 \leq i \leq k\}$ is a set of $k$ distinct elements -- write $A_0$ for this set.

The remainder of the proof follows that of \cite[Theorem 4]{ellgij::0} very closely.  Let $p \in \left(\frac{1}{2},1\right]$ be a constant to be optimised later with $np$ an odd integer, and suppose that $k \geq 2^{H(p)n+1}$.  Write $d:=np-1$ (which is an even integer) and so by (\ref{eqn.entropy})
\begin{align*}
\frac{1}{2}k\geq 2^{H(p)n} & \geq \sum_{r=d+1}^n{\binom{n}{r}} = 2^n - \sum_{r=0}^d{\binom{n}{r}}.
\end{align*}
Writing $V:=S_n^d \cap \{F:\F_2^n \rightarrow \F_2 : F(x)=0_{\F_2} \text{ for all }x \in \neg A\}$ and rearranging the above we get
\begin{equation*}
\dim V \geq \sum_{r=0}^d{\binom{n}{r}} - |\neg A| \geq |A| - \frac{1}{2}k.
\end{equation*}
Let $P \in S_n^d$ be a polynomial that is $0_{\F_2}$ on $\neg A$ and of maximal support and write $\Sigma$ for the support of $P$.  If $|\Sigma| < \dim V$ then there would be some $Q \in V$ not identically $0_{\F_2}$ with $Q(x)=0_{\F_2}$ for all $x \in \Sigma$, so that $Q+P$ would have larger support.  We conclude that the support of $P$ has size at least $|A| - \frac{1}{2}k$ and so includes at least half of $A_0$.

Write $I:=\{1 \leq i \leq k: P(s_i+t_i) \neq 0_{\F_2}\}$ so that $|I| \geq \frac{1}{2}k$.  The matrix $(P(s+t))_{s \in S,t\in T}$ includes the rows $(P(s_i+t_j))_{j=1}^{k}$ for $i \in I$.  If $i \in I$ then $P(s_i+t_i) \neq 0_{\F_2}$ and $P(s_i+t_j) = 0_{\F_2}$ for all $1 \leq j <i$, and so the rows generate a space of dimension at least $|I|$, and hence $\rk_{\F_2}(P(s+t))_{s \in S,t\in T} \geq \frac{1}{2}k$.

On the other hand (as in \cite[(1)]{ellgij::0}) there are constants $c_{m,m'}\in \F_2$ such that
\begin{align*}
P(x+y)& =\sum_{m,m' \in M_n^d:\deg mm' \leq d}{c_{m,m'}m(x)m'(y)}\\
& = \sum_{m \in M_n^{\frac{d}{2}}}{m(x)F_m(y)} + \sum_{m' \in M_n^{\frac{d}{2}}}{F_{m'}'(x)m'(y)} \text{ for all }x,y \in \F_2^n.
\end{align*}
It follows from (\ref{eqn.entropy}) again that
\begin{align*}
\rk_{\F_2} (P(s+t))_{s \in S,t \in T} \leq 2\dim S_n^{\frac{d}{2}}& \leq 2\sum_{r=0}^{\frac{d}{2}}{\binom{n}{r}} = 2\sum_{r=n-\frac{d}{2}}^n{\binom{n}{r}} \leq 2^{1+H\left(1-\frac{p}{2}+\frac{1}{2n}\right)n},
\end{align*}
and since $H$ is decreasing on $\left[\frac{1}{2},1\right]$ we conclude that
\begin{equation*}
k \leq \max\left\{2^{H(p)n+1},2^{H\left(1-\frac{p}{2}\right)n + 2}\right\}.
\end{equation*}
We get the result on putting $p:=\frac{2}{3}+O(n^{-1})$, so that $H(p) =H\left(1-\frac{p}{2}\right)+o(1)$.
\end{proof}
It may be worth noting that (\ref{eqn.problem}) is a special case of a more general fact.  Suppose that $s,t \in (\F_2^\infty)^k$ witness the $k$-order property in $A$ and consider the $\F_2^\infty$-valued matrix $M$ with $M_{ij}:=s_i+t_j$.  If $1 \leq i\leq j < i' \leq j' \leq k$ are such that $M_{ij}=M_{i'j'}$, then
\begin{equation*}
M_{i'j}=s_{i'}+t_{j}  = s_i + M_{ij} + t_{j'} + M_{i'j'} = M_{ij'} + 2M_{ij} = M_{ij'},
\end{equation*}
which contradicts the fact that $1_A(M_{i'j})=0 \neq 1 = 1_A(M_{ij'})$.  Put another way we have $M_{ij}\neq M_{i'j'}$ whenever $1 \leq i\leq j < i' \leq j' \leq k$.  One might hope that this condition alone requires $M$ to take many different values (and hence $A$ to be large compared with $k$).  However, it is possible to construct a matrix satisfying this property (and having distinct values in every row and column) using $O(k\log k)$ distinct elements.

Our argument is very similar to the arguments of Dvir and Edelman \cite{dviede::0} who apply the Croot-Lev-Pach method to examine the rigidity \cite[Definition, \S6]{val::0} of certain random matrices.  The matrix we have to consider is the `all-ones' upper triangular matrix and as it happens the rigidity of this has been explicitly calculated in \cite[Theorem 1]{pudvav::0}.  It is very natural to imagine more can be made of this structure.

With these two lemmas we can prove our main result.
\begin{proof}[Proof of Theorem \ref{thm.main}]
Suppose that $A$ has the $k$-order property.  Let $K:=N/k$ and apply Lemma \ref{lem.deco} with $l:=\left\lfloor \frac{k}{4}\right \rfloor$ to get $n \in \N$ and a set $A'\subset \F_2^n$ that has the $\frac{1}{2}k$-order property where $2^n \leq O(K^{15}k)$.  Then apply Lemma \ref{lem.dense} to $A'$ to get an absolute $c_0 \in (0,1)$ such that $\frac{1}{2}k=O(K^{15(1-c_0)}k^{1-c_0})$.  The result follows with $c=c_0/(15-14c_0)$.
\end{proof}
The extension of Theorem \ref{thm.main} to groups of bounded exponent seems interesting, though there the constant $c$ would have to depend on the exponent since Proposition \ref{prop.ap} extends from integers to any Abelian group  if we add the hypothesis $|A+A|=2|A|-1$.

The proof of Lemma \ref{lem.deco} extends easily, as does much of the proof of Lemma \ref{lem.dense}.  In particular, the Croot-Lev-Pach method has been extended to groups of bounded exponent (see \emph{e.g.} the proof of \cite[Theorem A]{blachu::0}).  However, (\ref{eqn.problem}) relies on working in characteristic $2$ and this would need to be replaced in the more general setting.

It remains to prove Proposition \ref{prop.lower}; we shall show the following explicit version.
\begin{proposition}
For all $N \in \N$ there is a set $A \subset \F_2^\infty$ of size $N$ that is not $N^{\frac{1}{2-c}-o(1)}$-stable where $c=\log_8\left(1+\frac{5-2\sqrt{2}}{3+2\sqrt{2}}\right)=0.152\dots$
\end{proposition}
\begin{proof}
Write $G:=\F_2^{\infty}$; let $l \in \N$ be a parameter to be optimised later; let $R:=\binom{2l}{l}$; and let $S_1,\dots,S_R$ be an enumeration of the subsets of $[2l]$ of size $l$ such that $S_r \cup S_{R+1-r}=[2l]$ for all $r \in \{1,\dots, R\}$.  (For example, proceed iteratively: first select $S_1$ arbitrarily, then put $S_R:=[2l] \setminus S_1$; select $S_2$ from what remains and put $S_{R-1}:=[2l]\setminus S_2$; \emph{etc}.)  Write
\begin{equation*}
V_S:=\left\{x \in G: x_i=0 \text{ whenever }i \not \in S\right\} \text{ for }S \subset [2l].
\end{equation*}
Let $u_1,\dots,u_R,w_1,\dots,w_R \subset G$ be such that
\begin{equation}\label{eqn.spread}
u_i+w_j+V_{[2l]} \text{ are pairwise disjoint for }1 \leq i,j \leq R.
\end{equation}
(For example by selecting greedily.) For each $1\leq i \leq R$, let $v_1^{(i)},\dots,v_{2^{l}}^{(i)}$ be an enumeration of $V_{S_i}$.  Finally, let
\begin{equation*}
\Delta_{i,j}:=\begin{cases}V_{S_{i}}+V_{S_{R+1-j}} & \text{ if } i <j\\
\left\{v_m^{(i)} + v_n^{(R+1-i)}: 1 \leq m \leq n \leq 2^{l}\right\} & \text{ if }i =j\end{cases}
\end{equation*}
and
\begin{equation*}
A:=\bigcup_{1 \leq i \leq j \leq R}{\left(u_i+w_j + \Delta_{i,j}\right)}.
\end{equation*}
Now
\begin{align*}
|A| \leq 2^{2l}\sum_{1 \leq i,j\leq R}{2^{-|S_i \cap S_j|}} &= 2^{2l}\sum_{S,T \subset [2l], |S|=|T|=l}{2^{-|S\cap T|}}\\ & = 2^{2l}\sum_{s=0}^l{\binom{2l}{s}\binom{2l-s}{l-s}\binom{l}{l-s}2^{-s}}\\ &=\binom{2l}{l}2^{2l}\sum_{s=0}^l{\binom{l}{s}^22^{-s}}\\
& \leq 2^{4l}\left(\sum_{s=0}^l{\binom{l}{s}\sqrt{2}^{-s}}\right)^2 = 2^{4l}\left(1+\frac{1}{\sqrt{2}}\right)^{2l}.
\end{align*}
Now let
\begin{equation*}
s=(s_1,\dots,s_{R2^l}):=\left(u_1+v_1^{(1)},\dots,u_1+v_{2^l}^{(1)},u_2+v_1^{(2)},\dots,u_R+v_{1}^{(R)},\dots,u_R+v_{2^l}^{(R)}\right)
\end{equation*}
and
\begin{equation*}
t=(t_1,\dots,t_{R2^l}):=\left(w_1+v_1^{(R)},\dots,w_1+v_{2^l}^{(R)},w_2+v_1^{(R-1)},\dots,w_R+v_{1}^{(1)},\dots,w_R+v_{2^l}^{(1)}\right).
\end{equation*}
Suppose that $1 \leq i,j \leq R2^l$, and let $1 \leq b,b' \leq R$ and $1 \leq a,a' \leq 2^l$ be the unique integers such that $i=a+2^l(b-1)$ and $j=a' + 2^l(b'-1)$, so that
\begin{equation}\label{eqn.diver}
s_i+t_j =u_b + w_{b'} +v_a^{(b)} + v_{a'}^{(R+1-b')}.
\end{equation}
If $i \leq j$ then either
\begin{enumerate}
\item $b < b'$, in which case $s_i+t_j \in u_b+w_{b'} + V_{S_{b}} + V_{S_{R+1-b'}} \subset A_0 \subset A$;
\item or $b=b'$, in which case $a \leq a'$ and $s_i+t_j  \in u_b+w_b + \Delta_{b,b} \subset A$.
\end{enumerate}
In the other direction, suppose $s_i+t_j \in A$.  Then by definition of $A$ there are some $1 \leq i' \leq j' \leq R$ such that $s_i+t_j \in u_{i'}+w_{j'}+\Delta_{i',j'}$.  By (\ref{eqn.diver}) we have $s_i+t_j \in u_b+w_{b'}+V_{[2l]}$ and so by (\ref{eqn.spread}) we see that $i'=b$ and $j'=b'$.  Either
\begin{enumerate}
\item $b \neq b'$, in which case the fact that $b=i' \leq j'=b'$ implies that $b<b'$ and hence $i<j$;
\item or $b=b'$, in which case $i'=j'=b$ and there are elements $1 \leq m' \leq n' \leq 2^l$ such that $s_i+t_j = u_b+w_b+ v_{m'}^{(b)} + v_{n'}^{(R+1-b)}$, which substituted into (\ref{eqn.diver}) gives $v_a^{(b)} + v_{a'}^{(R+1-b)}=v_{m'}^{(b)} + v_{n'}^{(R+1-b)}$.  Since $S_b \cap S_{R+1-b} =\emptyset$ it follows that $a=m'$ and $a'=n'$ so that $i \leq j$ since $m' \leq n'$.
\end{enumerate}
It follows that $A$ has the $R2^l$-order property, and taking $l$ suitably in terms of $N$ gives the result.
\end{proof}

\section*{Acknowledgements}

My thanks to Julia Wolf for useful conversations on this topic, and to an anonymous referee for careful reading of the paper and in particular highlighting an error in the proof of Lemma \ref{lem.deco}.

\bibliographystyle{halpha}

\bibliography{references}

\newcommand{\etalchar}[1]{$^{#1}$}
\begin{thebibliography}{BCC{\etalchar{+}}17}
\expandafter\ifx\csname url\endcsname\relax
  \def\url#1{\texttt{#1}}\fi
\expandafter\ifx\csname doi\endcsname\relax
  \def\doi#1{\burlalt{doi:#1}{http://dx.doi.org/#1}}\fi
\expandafter\ifx\csname urlprefix\endcsname\relax\def\urlprefix{URL }\fi
\expandafter\ifx\csname href\endcsname\relax
  \def\href#1#2{#2}\fi
\expandafter\ifx\csname burlalt\endcsname\relax
  \def\burlalt#1#2{\href{#2}{#1}}\fi

\bibitem[BCC{\etalchar{+}}17]{blachu::0}
J.~Blasiak, T.~Church, H.~Cohn, J.~A. Grochow, E.~Naslund, W.~F. Sawin, and
  C.~Umans.
\newblock On cap sets and the group-theoretic approach to matrix
  multiplication.
\newblock {\em Discrete Anal.}, pages Paper No. 3, 27, 2017,
  \burlalt{arXiv:1605.06702}{http://arxiv.org/abs/arXiv:1605.06702}.
\newblock \doi{10.19086/da.1245}.

\bibitem[CLP17]{crolevpac::}
E.~S. Croot, V.~F. Lev, and P.~P. Pach.
\newblock Progression-free sets in {$\mathbb{Z}_4^n$} are exponentially small.
\newblock {\em Ann. of Math. (2)}, 185(1):331--337, 2017,
  \burlalt{arXiv:1605.01506}{http://arxiv.org/abs/arXiv:1605.01506}.
\newblock \doi{10.4007/annals.2017.185.1.7}.

\bibitem[Con18]{con::4}
G.~Conant.
\newblock Multiplicative structure in stable expansions of the group of
  integers.
\newblock {\em Illinois J. Math.}, 62(1-4):341--364, 2018,
  \burlalt{arXiv:1704.00105}{http://arxiv.org/abs/arXiv:1704.00105}.
\newblock \doi{10.1215/ijm/1552442666}.

\bibitem[Con19]{con::2}
G.~Conant.
\newblock Stability and sparsity in sets of natural numbers.
\newblock {\em Israel J. Math.}, 230(1):471--508, 2019,
  \burlalt{arXiv:1701.01387}{http://arxiv.org/abs/arXiv:1701.01387}.
\newblock \doi{10.1007/s11856-019-1835-0}.

\bibitem[CZ01]{caszie::}
E.~Casanovas and M.~Ziegler.
\newblock Stable theories with a new predicate.
\newblock {\em The Journal of Symbolic Logic}, 66(3):1127--1140, 2001.
\newblock \doi{10.2307/2695097}.

\bibitem[DE17]{dviede::0}
Z.~Dvir and B.~Edelman.
\newblock Matrix rigidity and the {C}root-{L}ev-{P}ach lemma.
\newblock {\em Theory of Computing (to appear)}, August 2017,
  \burlalt{arXiv:1708.01646}{http://arxiv.org/abs/arXiv:1708.01646}.

\bibitem[EG17]{ellgij::0}
J.~S. Ellenberg and D.~Gijswijt.
\newblock On large subsets of {$\mathbb{F}_q^n$} with no three-term arithmetic
  progression.
\newblock {\em Annals of Mathematics}, 185(1):339--343, 2017,
  \burlalt{arXiv:1605.09223}{http://arxiv.org/abs/arXiv:1605.09223}.
\newblock \doi{10.4007/annals.2017.185.1.8}.

\bibitem[GR07]{greruz::0}
B.~J. Green and I.~Z. Ruzsa.
\newblock Freiman's theorem in an arbitrary {A}belian group.
\newblock {\em J. Lond. Math. Soc. (2)}, 75(1):163--175, 2007,
  \burlalt{arXiv:math/0505198}{http://arxiv.org/abs/arXiv:math/0505198}.
\newblock \doi{10.1112/jlms/jdl021}.

\bibitem[Gre05]{gre::9}
B.~J. Green.
\newblock Finite field models in additive combinatorics.
\newblock In {\em Surveys in combinatorics 2005}, volume 327 of {\em London
  Math. Soc. Lecture Note Ser.}, pages 1--27. Cambridge Univ. Press, Cambridge,
  2005, \burlalt{arXiv:math/0409420}{http://arxiv.org/abs/arXiv:math/0409420}.
\newblock \doi{10.1017/CBO9780511734885.002}.

\bibitem[Mac71]{mac::1}
A.~Macintyre.
\newblock On {$\omega_1$}-categorical theories of {A}belian groups.
\newblock {\em Fundamenta Mathematicae}, 70(3):253--270, 1971.
\newblock \doi{10.4064/fm-70-3-253-270}.

\bibitem[MS77]{macslo::0}
F.~J. MacWilliams and N.~J.~A. Sloane.
\newblock {\em The theory of error-correcting codes. {I}}.
\newblock North-Holland Publishing Co., Amsterdam-New York-Oxford, 1977.
\newblock North-Holland Mathematical Library, Vol. 16.

\bibitem[Pet12]{pet::}
G.~Petridis.
\newblock New proofs of {P}l{\"u}nnecke-type estimates for product sets in
  groups.
\newblock {\em Combinatorica}, 32(6):721--733, 2012,
  \burlalt{arXiv:1101.3507}{http://arxiv.org/abs/arXiv:1101.3507}.
\newblock \doi{10.1007/s00493-012-2818-5}.

\bibitem[PS18]{palskl::}
D.~Palac\'{\i}n and R.~Sklinos.
\newblock On superstable expansions of free {A}belian groups.
\newblock {\em Notre Dame J. Form. Log.}, 59(2):157--169, 2018,
  \burlalt{arXiv:1405.0568}{http://arxiv.org/abs/arXiv:1405.0568}.
\newblock \doi{10.1215/00294527-2017-0023}.

\bibitem[PV91]{pudvav::0}
P.~Pudl\'{a}k and Z.~Vav\v{r}\'{i}n.
\newblock Computation of rigidity of order {$n^2/r$} for one simple matrix.
\newblock {\em Comment. Math. Univ. Carolin.}, 32(2):213--218, 1991.
\newblock \doi{10338.dmlcz/116958}.

\bibitem[She71]{she::0}
S.~Shelah.
\newblock Stability, the f.c.p., and superstability; model theoretic properties
  of formulas in first order theory.
\newblock {\em Ann. Math. Logic}, 3(3):271--362, 1971.
\newblock \doi{10.1016/0003-4843(71)90015-5}.

\bibitem[{Sis}18]{sis::1}
O.~{Sisask}.
\newblock {Convolutions of sets with bounded VC-dimension are uniformly
  continuous}.
\newblock {\em ArXiv e-prints}, February 2018,
  \burlalt{arXiv:1802.02836}{http://arxiv.org/abs/arXiv:1802.02836}.

\bibitem[TV06]{taovu::}
T.~C. Tao and V.~H. Vu.
\newblock {\em Additive combinatorics}, volume 105 of {\em Cambridge Studies in
  Advanced Mathematics}.
\newblock Cambridge University Press, Cambridge, 2006.
\newblock \doi{10.1017/CBO9780511755149}.

\bibitem[TW18]{terwol::1}
C.~{Terry} and J.~{Wolf}.
\newblock {Quantitative structure of stable sets in finite Abelian groups}.
\newblock {\em ArXiv e-prints}, May 2018,
  \burlalt{arXiv:1805.06847}{http://arxiv.org/abs/arXiv:1805.06847}.

\bibitem[TW19]{terwol::0}
C.~Terry and J.~Wolf.
\newblock Stable arithmetic regularity in the finite field model.
\newblock {\em Bulletin of the London Mathematical Society}, 51(1):70--88,
  2019, \burlalt{arXiv:1710.02021}{http://arxiv.org/abs/arXiv:1710.02021}.
\newblock \doi{10.1112/blms.12211}.

\bibitem[Val77]{val::0}
L.~G. Valiant.
\newblock Graph-theoretic arguments in low-level complexity.
\newblock pages 162--176. Lecture Notes in Comput. Sci., Vol. 53, 1977.
\newblock \doi{10.1007/3-540-08353-7_135}.

\bibitem[Wol15]{wol::3}
J.~Wolf.
\newblock Finite field models in arithmetic combinatorics---ten years on.
\newblock {\em Finite Fields Appl.}, 32:233--274, 2015.
\newblock \doi{10.1016/j.ffa.2014.11.003}.

\end{thebibliography}

\end{document}